\documentclass[runningheads]{llncs}
\usepackage{hyperref}
\usepackage{graphicx}
\usepackage{todonotes}
\usepackage{xcolor}
\usepackage{amsmath, amssymb, bbm}
\usepackage{enumitem}

\newcommand{\flt}{\mathrm{Flt}}
\newcommand{\R}{\mathbb{R}}
\newcommand{\Z}{\mathbb{Z}}
\newcommand{\lw}{\operatorname{lw}}
\newcommand{\conv}{\operatorname{conv}}
\newcommand{\zero}{\mathbf{0}}
\newcommand{\unitvec}{\mathbf{e}}
\newcommand{\one}{\mathbf{1}}
\newcommand{\shift}[2]{{\overrightarrow{#1}}^{#2}} 

\makeatletter
\renewcommand{\@Opargbegintheorem}[4]{%
  #4\trivlist\item[\hskip\labelsep{#3#2\@thmcounterend}]}
\makeatother

\begin{document}

\title{Lattice-free simplices with lattice width $2d - o(d)$}
\author{Lukas Mayrhofer \and Jamico Schade \and Stefan Weltge}
\institute{Technical University of Munich, Germany\\
\email{\{lukas.mayrhofer,jamico.schade,weltge\}@tum.de}}

\maketitle

\begin{abstract}
The Flatness theorem states that the maximum lattice width $\flt(d)$ of a $d$-dimensional lattice-free convex set is finite.
It is the key ingredient for Lenstra's algorithm for integer programming in fixed dimension, and much work has been done to obtain bounds on $\flt(d)$.
While most results have been concerned with upper bounds, only few techniques are known to obtain lower bounds.
In fact, the previously best known lower bound $\flt(d) \ge 1.138d$ arises from direct sums of a $3$-dimensional lattice-free simplex.

In this work, we establish the lower bound $\flt(d) \ge 2d - O(\sqrt{d})$, attained by a family of lattice-free simplices.
Our construction is based on a differential equation that naturally appears in this context.

Additionally, we provide the first local maximizers of the lattice width of $4$- and $5$-dimensional lattice-free convex bodies.

\end{abstract}

\section{Introduction}
A convex body in $\R^d$ is called \emph{lattice-free} if it does not contain any integer points in its interior.
Lattice-free convex bodies appear in many important works concerning the theory of integer programming.
They are central objects in cutting plane theory~\cite{balas1971intersection,andersen2007inequalities,borozan2009minimal,del2016relaxations,basu2015geometric} and can been used as certificates of optimality in convex integer optimization~\cite{blair1984constructive,blair1985constructive,wolsey1981b,wolsey1981integer,moran2012strong,baes2016duality,basu2017optimality}.
Moreover they play a crucial role in Lenstra's algorithm~\cite{lenstra1983integer} for integer programming in fixed dimension.

A fundamental property of lattice-free convex bodies is that they are ``flat'' with respect to the integer lattice.
The \emph{lattice width} of a convex body $K \subseteq \R^d$ is
\[
    \lw(K) := \min \left \{ \max_{x \in K} \langle x,y \rangle - \min_{x \in K} \langle x,y \rangle : y \in \Z^d \setminus \{ \zero \} \right \},
\]
where $\langle \cdot,\cdot \rangle$ denotes the standard scalar product.
The famous \emph{Flatness theorem}, first proved by Khinchine~\cite{khinchine1948quantitative}, states that the maximum lattice width $\flt(d)$ of a $d$-dimensional lattice-free convex body is finite.
It has been applied in several results in mixed-integer programming (e.g., \cite{lenstra1983integer,averkov2018approximation,dash2014lattice,fukasawa2019not,cevallos2018lifting}) and much work has been done to obtain bounds on $\flt(d)$.
Combining the work of Banaszczyk, Litvak, Pajor, Szarek~\cite{banaszczyk1999flatness} and Rudelson~\cite{rudelson}, the currently best upper bound is $\flt(d) = O(d^{4/3} \operatorname{polylog}(d))$, and it is conjectured that $\flt(d) = \Theta(d)$ holds~\cite{banaszczyk1999flatness}.

While most results have been concerned with upper bounds, only few techniques are known to obtain lower bounds.
It is easy to see that $\flt(d) \ge d$ holds by observing that $\Delta := \conv \{ \zero, d \cdot \unitvec_1,\dots,d \cdot \unitvec_d\}$ is lattice-free and satisfies $\lw(\Delta) = d$.
This bound has been improved for small dimensions:
For $d=2$, Hurkens~\cite{hurkens1990blowing} proved that
\[
    \flt(2) = 1 + 2/\sqrt{3} \approx 1.08d
\]
holds, and this is already the last dimension for which $\flt(d)$ is known.
For $d=3$, Codenotti \& Santos~\cite{codenotti2020hollow} recently constructed a $3$-dimensional lattice-free simplex of lattice width $2 + \sqrt{2}$, showing that
\[
    \flt(3) \ge 2 + \sqrt{2} \approx 1.14d
\]
holds.
Averkov, Codenotti, Macchia \& Santos~\cite{averkov2021local} showed that this simplex is a local maximizer of the lattice width among $3$-dimensional lattice-free convex bodies.
In fact, Codenotti \& Santos conjectured that it is actually a global maximizer, i.e., that the above bound is tight.

It is possible to lift these examples to higher dimensions:
Using the notion of a direct sum, one can show that $\flt(d_1 + d_2) \ge \flt(d_1) + \flt(d_2)$ holds for all positive integers $d_1,d_2$, see~\cite[Prop.~1.4]{codenotti2020hollow}.
However, prior to this work, $d$-dimensional lattice-free convex bodies with lattice width strictly larger than $1.14d$ were not known.

In this work, we establish new lower bounds on $\flt(d)$.
In terms of small dimensions, we obtain the first local maximizers of the lattice width among $d$-dimensional lattice-free convex bodies for $d=4$ and $d=5$, implying
\[
    \flt(4) \ge 2 + 2 \sqrt{1 + 2/\sqrt{5}} \approx 1.19d
    \quad \text{and} \quad
    \flt(5) \ge 5 + \frac{2}{\sqrt{3}} \approx 1.23d.
\]
Unfortunately, it becomes inherently more difficult to obtain local maximizers of the lattice width in larger dimensions.
However, we show that for each $d$, it is still possible to construct lattice-free simplices with strictly increasing ratios of lattice width to dimension.
Our main result is the following:
\begin{theorem}
    \label{thmmain}
    There exist $d$-dimensional lattice-free simplices $(\Delta_d)_{d \geq 2}$ with lattice width $2d - O(\sqrt{d})$.
\end{theorem}
In particular, we see that
\[
    \flt(d) \ge 2d - O(\sqrt{d})
\]
holds.

In contrast to the previously mentioned lower bounds, Theorem~\ref{thmmain} is not based on lifting small-dimensional examples to higher dimensions, but yields explicit constructions of lattice-free simplices for all dimensions.
An interesting aspect of our construction is that it arises by considering the problem from an infinite-dimensional point of view.
In fact, the facet-defining normal vectors of our simplices are discretizations of the solution to a differential equation that naturally appears within our approach (see Section~\ref{secBackground}).

Our paper is organized as follows.
In Section~\ref{secMain}, we describe the simplices mentioned in Theorem~\ref{thmmain}, prove that they are lattice-free, and determine their lattice widths.
The background on how our construction was obtained is provided in Section~\ref{secBackground}.
In Section~\ref{secLocalMax}, we present the local maximizers of the lattice width in dimensions $4$ and $5$.
We conclude with some open questions in Section~\ref{secQuestions} and comment on difficulties for obtaining local maximizers in dimensions $d \ge 6$.

\subsection*{Acknowledgements}
The authors would like to thank Gennadiy Averkov and Paco Santos for valuable feedback and discussions on earlier stages of this work, and Amitabh Basu for discussions about applications of the Flatness theorem within optimization.
This work was supported by the Deutsche Forschungsgemeinschaft (DFG, German Research Foundation), project number 451026932.

\section{Lattice-free simplices of large lattice width}
\label{secMain}

In this part, we will present $d$-dimensional lattice-free simplices with lattice width $2d - O(\sqrt{d})$.
Our construction yields highly symmetric simplices, at least when viewed as subsets of
\[
    H := \{ x \in \R^{d+1} : \langle \one,x \rangle = 1 \}.
\]
In fact, we will consider a simplex $\Delta \subseteq H$ and see that its projection $\pi(\Delta) \subseteq \R^d$ onto the first $d$ coordinates has the desired properties.
The simplex is given by
\[
    \Delta := \left \{ x \in H : \langle \shift{a}{i}, x \rangle \le a_{d+1} \text{ for } i = 0,1,\dots,d \right \},
\]
where $a \in \R^{d+1}$ is defined via
\[
    a_i := \delta^{i-1} - 1
\]
for every $i \in [d+1] := \{1,\dots,d+1\}$ with
\[
    \delta := \left( 1 - \sqrt{\frac{2}{d+1}} \right)^{-1},
\]
and $\shift{a}{i}$ arises from $a$ by cyclically shifting all entries $i$ positions to the right.
During this section, we will prove the following statements.

\begin{enumerate}[label=(\Alph*),leftmargin=3em]
    \item \label{propDimension} $\Delta$ is a $d$-dimensional simplex.
    \item \label{propLatticeFree} $\Delta$ does not contain an integer point in its relative interior.
    \item \label{propWidth} For every $c \in \Z^{d + 1} \setminus \{ \lambda \cdot \one : \lambda \in \R \}$ we have
        \begin{equation}
            \label{eqd3f48h}
            \max_{x \in \Delta} \langle c, x \rangle - \min_{x \in \Delta} \langle c, x \rangle \ge
            \frac{d \cdot \delta^{d+ 2} - \delta^{d + 1} - (d + 1) \cdot \delta^d + \delta + 1}{(\delta - 1) \cdot (\delta^{d + 1} - 1)}.
        \end{equation}
\end{enumerate}

Let us first show that these claims indeed imply our main result.
\begin{proof}[Proof of Theorem~\ref{thmmain}]
    Since $\Delta$ is a $d$-dimensional simplex by~\ref{propDimension}, the same holds for $\pi(\Delta)$.
    Suppose that $\pi(\Delta)$ is not lattice-free, i.e., there exists a point $x' \in \Z^d$ in the interior of $\pi(\Delta)$.
    The integer point $x = (x_1',\dots,x_d',1 - x'_1 - \dots - x'_d)$ is then contained in the relative interior of $\Delta$, a contradiction to~\ref{propLatticeFree}.

    To obtain a lower bound on the lattice width of $\pi(\Delta)$, let $\alpha$ denote the right-hand side of~\eqref{eqd3f48h}.
    Consider any $c' \in \Z^d \setminus \{ \zero \}$ and note that $c := (c'_1,\dots,c'_d,0)$ satisfies $c \in \Z^{d + 1} \setminus \{ \lambda \cdot \one : \lambda \in \R \}$.
    Thus, by~\ref{propWidth} we obtain
    \[
        \max_{x \in \pi(\Delta)} \langle c', x \rangle - \min_{x \in \pi(\Delta)} \langle c', x \rangle
        =
        \max_{x \in \Delta} \langle c, x \rangle - \min_{x \in \Delta} \langle c, x \rangle
        \ge \alpha.
    \]
    In particular, we see that $\lw(\pi(\Delta)) \ge \alpha$ holds.
    Note that
    \[
        \alpha > \frac{d \cdot \delta^{d+ 2} - \delta^{d + 1} - (d + 1) \cdot \delta^d}{(\delta - 1) \cdot \delta^{d + 1}} = \frac{d \cdot \delta^2 - \delta - (d + 1)}{(\delta - 1) \cdot \delta}
    \]
    holds, where the inequality follows from $\delta > 1$.
    Substituting $k = \sqrt{(d+1)/2}$, we get $d = 2 k^2 - 1$ and $\delta = \frac{k}{k - 1}$, and hence
    \begin{align*}
        \alpha & > \frac{(2k^2 - 1) \cdot \left(\frac{k}{k-1}\right)^2 - \frac{k}{k - 1} - 2k^2}{\frac{1}{k - 1} \cdot \frac{k}{k - 1}} \\
        &= 2(2k^2 - 1) - 4k + 3 = 2d - 4k + 3
        = 2d - \sqrt{8d+8} + 3. \quad \qed
    \end{align*}
\end{proof}

In the proofs of~\ref{propDimension} and~\ref{propLatticeFree} we will make use of the following auxiliary facts.

\begin{lemma}
\label{NicePermutationsExist}
For every $x \in \R^{d+1}$ with $\langle \one,x \rangle \ge 0$ there exists some $\ell$ such that $\shift{x}{\ell}_1 + \dots + \shift{x}{\ell}_j \le \langle \one,x \rangle$ holds for all $j \in [d]$.
Moreover, if $\langle \one,x \rangle > 0$, then each of these inequalities is strict.
\end{lemma}

\begin{proof}
For each $j \in [d+1]$ we define $S_j = x_1 + \dots + x_j$.
Note that $S_{d+1} = \langle \one,x \rangle$.
Let $k \in [d+1]$ denote the smallest index such that $S_k = \max \{S_1,\dots,S_{d+1}\}$ holds.
We will show that $\ell := d + 1 - k$ satisfies the claim.

Recall that $\shift{x}{\ell} = (x_{k+1}, x_{k+2}, \dots, x_{d+1}, x_1, \dots, x_k)$ and let $j \in [d]$.
If $j \le d + 1 - k$, then we have
\[
    \shift{x}{\ell}_1 + \dots + \shift{x}{\ell}_j = x_{k+1} + \dots + x_{k+j} = S_{k+j} - S_k \le 0 \le \langle \one,x \rangle,
\]
where the inequality $S_{k + j} - S_k \le 0$ holds due to the choice of $k$.
If $j \ge d + 2 - k$, then
\begin{align*}
    \shift{x}{\ell}_1 + \dots + \shift{x}{\ell}_j & = x_{k+1} + \dots + x_{d+1} + x_1 + \dots + x_{j + k - (d + 1)} \\
    & = S_{d+1} - S_k + S_{j + k - (d + 1)}.
\end{align*}
Since $j \le d$, we have $j + k - (d + 1) < k$ and hence by the choice of $k$ we see that $S_{j + k - (d + 1)} < S_k$ holds, which yields $\shift{x}{\ell}_1 + \dots + \shift{x}{\ell}_j < S_{d+1} = \langle \one,x \rangle$. \qed
\end{proof}

\begin{lemma}
\label{lemBounded}
    For every $x \in \R^{d+1} \setminus \{\zero\}$ with $\langle \one,x \rangle = 0$ there exists some $i$ such that $\langle \shift{a}{i},x \rangle > 0$.
\end{lemma}
\begin{proof}
    Let $\ell$ be as in Lemma~\ref{NicePermutationsExist}.
    Defining $y := \shift{x}{\ell}$ we have that $S_j := y_1 + \dots + y_j \le 0$ holds for all $j \in [d]$.
    Recall that $S_{d+1} := y_1 + \dots + y_{d+1} = 0$ and that $y \ne \zero$.
    Thus, we must have $S_{j^*} < 0$ for some $j^* \in [d]$.
    We obtain
    \begin{align*}
        \langle a,y \rangle
        & = a_1 y_1 + a_2 y_2 + a_3 y_3 + \dots + a_{d+1} y_{d+1} \\
        & = a_1 S_1 + a_2 (S_2 - S_1) + a_3 (S_3 - S_2) + \dots + a_{d+1} (S_{d+1} - S_d) \\
        & = S_1 (a_1 - a_2) + S_2 (a_2 - a_3) + \dots + S_d (a_d - a_{d+1}) + a_{d+1} S_{d+1} \\
        & \ge S_{j^*} (a_{j^*} - a_{j^* + 1}) + a_{d+1} S_{d+1} \\
        & = S_{j^*} (a_{j^*} - a_{j^* + 1}) \\
        & > 0,
    \end{align*}
    where we used the fact that $a_j < a_{j+1}$ holds for every $j \in [d]$.
    The claim follows since $\langle \shift{a}{d+1-\ell},x\rangle = \langle a,\shift{x}{\ell} \rangle = \langle a,y \rangle$. \qed
\end{proof}

\begin{lemma}
\label{lemLatticeFree1}
    For every $x \in H \cap \Z^{d+1}$ there exists some $i$ such that $\langle \shift{a}{i},x \rangle \ge a_{d+1}$. In particular, this implies that $\Delta$ does not contain an integer point in its relative interior.
\end{lemma}
\begin{proof}
    We proceed similarly to the proof of Lemma~\ref{lemBounded}.
    We pick $\ell$ as in Lemma~\ref{NicePermutationsExist} and define $y := \shift{x}{\ell}$.
    Note that $S_j := y_1 + \dots + y_j < 1$ holds for all $j \in [d]$.
    Since $y$ is integer, we even have $S_j \le 0$ for all $j \in [d]$.
    Setting $S_{d+1} := y_1 + \dots + y_{d+1} = 1$ we hence obtain
    \begin{align*}
        \langle a,y \rangle
        & = a_1 y_1 + a_2 y_2 + a_3 y_3 + \dots + a_{d+1} y_{d+1} \\
        & = a_1 S_1 + a_2 (S_2 - S_1) + a_3 (S_3 - S_2) + \dots + a_{d+1} (S_{d+1} - S_d) \\
        & = S_1 (a_1 - a_2) + S_2 (a_2 - a_3) + \dots + S_d (a_d - a_{d+1}) + a_{d+1} S_{d+1} \\
        & \ge a_{d+1} S_{d+1} = a_{d+1},
    \end{align*}
    where the inequality holds since we have $a_j \le a_{j+1}$ for every $j \in [d]$ and $S_{d+1} = 1$. Again, the claim follows since $\langle \shift{a}{d+1-\ell},x\rangle = \langle a,\shift{x}{\ell} \rangle = \langle a,y \rangle$. \qed
\end{proof}

\begin{lemma}
    $\Delta$ is a $d$-dimensional simplex.
\end{lemma}
\begin{proof}
    The point $p := \frac{1}{d+1} \one$ is contained in $H$ and satisfies
    \[
        \langle \shift{a}{i},p \rangle = \frac{a_1 + \dots + a_{d+1}}{d+1} < a_{d+1}
    \]
    for all $i$.
    In particular, we see that $\Delta$ contains a ball within $H$ around $p$, and hence $\Delta$ is $d$-dimensional.

    Since $\Delta$ is defined by $d+1$ linear inequalities, it suffices to show that $\Delta$ is bounded in order to prove that it is a simplex.
    The recession cone of $\Delta$ is given by
    \[
         \text{recc}(\Delta) = \left \{ x \in \R^{d + 1}: \langle \one, x \rangle = 0, \langle \shift{a}{i}, x \rangle \le 0 \text{ for } i = 0,1,\dots,d \right \}.
    \]
    From Lemma~\ref{lemBounded}, it follows directly that $\text{recc}(\Delta) = \{ \zero \}$, proving that $\Delta$ is indeed bounded. \qed
\end{proof}

For the proof of~\ref{propWidth}, let us first determine the vertices of $\Delta$.
To this end, define the vector $v \in \R^{d+1}$ via
\begin{align*}
    v_1 & = \frac{1 - d \cdot \delta^d + (d - 2) \cdot \delta^{d + 1} + \delta^{d + 2}}{(\delta - 1) \cdot (\delta^{d + 1} - 1)}, \\
    v_2 = \dots = v_d & = \frac{(\delta - 1) \cdot \delta^d}{\delta^{d + 1} - 1}, \\
    v_{d + 1} & = \frac{- \delta + \delta^d + (d - 1) \cdot \delta^{d + 1} - (d - 1) \cdot \delta^{d + 2}}{(\delta - 1) \cdot (\delta^{d + 1} - 1)}. 
\end{align*}

\begin{lemma}
    $\Delta = \conv \left( \left\{\shift{v}{0}, \shift{v}{1}, \dots, \shift{v}{d+1} \right\} \right)$.
\end{lemma}
\begin{proof}
    For every $i$, we will show that $\shift{v}{i}$ is contained in $H$, satisfies all linear inequalities defining $\Delta$, and all but one even with equality.
    This shows that each $\shift{v}{i}$ is a vertex of $\Delta$.
    Since $\Delta$ is a simplex and all $\shift{v}{i}$ are distinct, we obtain the claim.
    To this end, due to the circulant structure of $\Delta$, we may assume that $\shift{v}{i} = v$ holds.
    First, observe that 
    \begin{align*}
        (\delta  - 1) \cdot (\delta^{d + 1} - 1) \cdot \langle \one, v \rangle
        & = (\delta  - 1) \cdot (\delta^{d + 1} - 1) \cdot (v_1 + (d - 1) \cdot v_2 + v_{d + 1}) \\
        & = \delta^{d + 2} - \delta^{d + 1} - \delta + 1
        = (\delta - 1) \cdot (\delta^{d + 1} - 1)
    \end{align*}
    holds, which shows that $v$ is contained in $H$.
    Next, note that we have
    \[
        \langle \overrightarrow{a}^0, v \rangle = v_2 \cdot \sum_{k = 1}^{d + 1} a_i + (v_1 - v_2) \cdot a_1 + (v_{d + 1} - v_2) \cdot a_{d+1}
    \]
    and
    \[
        \langle \overrightarrow{a}^i, v \rangle = v_2 \cdot \sum_{k = 1}^{d + 1} a_i + (v_1 - v_2) \cdot a_{d + 2 - i} + (v_{d + 1} - v_2) \cdot a_{d+1-i}
    \]
    for $i \in [d]$.
    To evaluate these expressions, observe that we have
    \begin{align*}
        (\delta - 1) \cdot (\delta^{d + 1} - 1) \cdot v_2 \cdot \sum_{k = 1}^{d + 1} a_i &= (\delta - 1)^2 \cdot \delta^d \cdot \sum_{k = 0}^d (\delta^k - 1) \\
        &= (\delta-1)^2 \cdot \delta^d \cdot \left( \frac{\delta^{d + 1} - 1}{\delta - 1} - (d + 1) \right) \\
        &=\delta^d \cdot (\delta - 1) \cdot \left((\delta^{d + 1} - 1) - (d + 1) \cdot (\delta - 1)\right),
    \end{align*}
    as well as 
    \[
        (\delta - 1) \cdot (\delta^{d + 1} - 1) \cdot (v_1 - v_2)
        = 1 - (d + 1) \cdot \delta^d + d \cdot \delta^{d + 1}
    \]
    and 
    \[
        (\delta - 1) \cdot (\delta^{d + 1} - 1) \cdot (v_{d + 1} - v_2)
        = - \delta + (d + 1) \cdot \delta^{d + 1} - d \cdot \delta^{d +2}.
    \]
    Thus, for $i=0$ we obtain
    \begin{align*}
        \underbrace{(\delta - 1) \cdot (\delta^{d + 1} - 1)}_{>0} \cdot \langle \overrightarrow{a}^0, v \rangle  
        &= (1 - \delta^{d + 1}) \cdot \delta \cdot \left( (d - 1) \cdot \delta^d - d \cdot \delta^{d - 1} + 1 \right) \\
        &= (1 - \delta^{d + 1}) \cdot \delta \cdot (\delta - 1) \cdot \left((d - 1) \cdot \delta^{d - 1} - \sum_{k = 0}^{d - 2} \delta^k \right) \\
        &= \underbrace{(1 - \delta^{d + 1})}_{<0} \cdot \underbrace{\delta \cdot (\delta - 1) \cdot \sum_{k = 0}^{d - 2} (\delta^{d - 1} - \delta^k)}_{> 0},
    \end{align*}
    which implies $\langle \overrightarrow{a}^0, v \rangle < 0 < a_{d + 1}$.
    For $i \in [d]$ we see that
    \begin{align*}
        (\delta - 1) \cdot (\delta^{d + 1} - 1) \cdot \langle \overrightarrow{a}^i, v \rangle
        &= (\delta^d - 1) \cdot (\delta - 1) \cdot (\delta^{d + 1} - 1) \\
        &= (\delta - 1) \cdot (\delta^{d + 1} - 1) \cdot a_{d + 1}
    \end{align*}
holds, and hence $\langle \overrightarrow{a}^i, v \rangle = a_{d + 1}$.
\qed
\end{proof}

Note that, in order to prove~\ref{propWidth}, it remains to show the following:

\begin{lemma}
\label{lemLatticeWidth}
    For every $c \in \Z^{d + 1} \setminus \{ \lambda \cdot \one : \lambda \in \R \}$ we have
    \begin{equation}
        \label{eqhude8j}
        \max_{x \in \Delta} \langle c, x \rangle - \min_{x \in \Delta} \langle c, x \rangle \ge v_1 - v_{d+1}.
    \end{equation}
\end{lemma}
\begin{proof}
    Since not all entries of $c$ are equal, there exists some $i$ such that $\shift{c}{i}_1 \ge \shift{c}{i}_j$ for all $j \in [d]$ and $\shift{c}{i}_1 > \shift{c}{i}_{d+1}$.
    Due to the circulant symmetry of $\Delta$, we may replace $c$ by $\shift{c}{i}$ without changing the left-hand side in~\eqref{eqhude8j}.
    Thus, we may assume that $c = \shift{c}{i}$ holds.
    Let $j \in [d]$ denote the smallest index such that $c_j > c_{j + 1}$ holds. Since $c_1$ is a maximal entry of $c$, it is clear that $c_j = c_1$ holds.
    Moreover, since $c$ is an integer vector, we see that
    \begin{equation}
        \label{eqhusd9k}
        c_1 = c_j \geq c_{j + 1} + 1 \text{ and } c_j = c_1 \geq c_{d + 1} + 1
    \end{equation}
    hold.
    We obtain
    \begin{align*}
        \max_{x \in \Delta} \langle c, x \rangle & - \min_{x \in \Delta} \langle c, x \rangle \\
        & \ge \langle c, v \rangle - \langle c, \shift{v}{j} \rangle
        = \langle c, v - \shift{v}{j} \rangle \\
        & = c_1 (v_1 - v_2) + c_j (v_2 - v_{d+1}) + c_{j+1} (v_2 - v_1) + c_{d+1} (v_{d+1} - v_2) \\
        & = (c_1 - c_{j+1}) (v_1 - v_2) + (c_j - c_{d+1}) (v_2 - v_{d+1}).
    \end{align*}
    We claim that $v_1 \ge v_2$ and $v_2 \ge v_{d+1}$, which, together with~\eqref{eqhusd9k}, implies
    \begin{align*}
        \max_{x \in \Delta} \langle c, x \rangle - \min_{x \in \Delta} \langle c, x \rangle
        & \ge (c_1 - c_{j+1}) (v_1 - v_2) + (c_j - c_{d+1}) (v_2 - v_{d+1}) \\
        & \ge (v_1 - v_2) + (v_2 - v_{d+1})
        = v_1 - v_{d+1},
    \end{align*}
    in which case we are done.
    To see that $v_1 \ge v_2$ and $v_2 \ge v_{d+1}$ hold, we use
    \begin{align*}
        \underbrace{(\delta - 1) \cdot (\delta^{d + 1} - 1)}_{> 0} \cdot (v_1 - v_2) &= 1 - d \cdot \delta^d + (d - 2) \cdot \delta^{d + 1} + \delta^{d + 2} - (\delta - 1)^2 \cdot \delta^d \\
        &= \underbrace{(\delta - 1)}_{> 0} \cdot \underbrace{\sum_{k = 0}^{d - 1} (\delta^d - \delta^k)}_{>0}
    \end{align*}
    and $v_2 - v_{d + 1} = \delta(v_1 - v_2)$. \qed
\end{proof}

\section{An infinite-dimensional view}
\label{secBackground}

In this section, we would like to provide some background on the construction of the simplices in the previous section.
It is inspired by the lattice-free simplices in~\cite{hurkens1990blowing} and~\cite{codenotti2020hollow}, which attain the largest (known) lattice widths in dimensions $2$ and $3$, respectively.
In fact, they can be also described in the form
\begin{equation}
    \label{eqForm}
    \left \{ x \in H : \langle \shift{a}{i}, x \rangle \le a_{d+1} \text{ for } i = 0,1,\dots,d \right \},
\end{equation}
for some vector $a \in \R^{d+1}$.

Such sets were also used by Doolittle, Katthän, Nill, Santos~\cite{doolittle2021empty} in the context of simplices whose integer points coincide with their vertices.
Within this setting, constructions of Sebő~\cite{sebHo1999introduction} were already based on highly-symmetric simplices.
Herr, Rehn, Schürmann~\cite{herr2015lattice} provide some more background on how symmetry interacts with lattice-freeness.

A convenient property of a set as in~\eqref{eqForm} is that it is lattice-free whenever the entries of $a$ are non-decreasing, see the proof of Lemma~\ref{lemLatticeFree1}.
Since we are limited to the hyperplane $H$, we may assume that $a_1 = 0$ and $a_{d+1} = 1$.
Every vertex of the above set is the cyclic shift of some vector $v \in \R^{d+1}$ satisfying
\begin{equation}
    \label{eqhusdok}
    \begin{bmatrix}
        1       & 1      & \cdots  & 1 & 1      \\
        a_2 & a_3    & \cdots & a_{d+1} & a_1    \\
        a_3 & a_4    & \cdots  & a_1 & a_2   \\
        \vdots  & \vdots &   & \vdots & \vdots \\
        a_{d+1}     & a_1    & \cdots & a_{d-1} & a_d
    \end{bmatrix}
    v = \one.
\end{equation}
In the examples from~\cite{hurkens1990blowing,codenotti2020hollow}, the lattice width is equal to $v_1 - v_{d+1}$, which is why we particularly focus on $\lambda = \frac{v_1 - v_{d+1}}{2d}$ in what follows.
(Note that we obtain simplices with $\lambda$ close to $1$.)
Let $C$ arise from the above matrix by deleting the first row as well as the first and last columns, and set $w = (v_2,\dots,v_d)$.
The above system is equivalent to
\[
    v_1 \begin{bmatrix}
        a_2 \\
        a_3 \\
        \vdots \\
        a_{d+1} \\
    \end{bmatrix}
    +
    C w
    +
    v_{d+1}
    \begin{bmatrix}
        a_1 \\
        a_2 \\
        \vdots \\
        a_d
    \end{bmatrix}
    =
    \one,
    \quad
    v_1 + \langle \one,w \rangle + v_{d+1} = 1.
\]
Substituting $v_1 = \frac{1}{2} (1 - \langle \one,w \rangle) + d\lambda$ and $v_{d+1} = \frac{1}{2} (1 - \langle \one,w \rangle) - d\lambda$, this leads to
\[
    \lambda \cdot d
    \left(
    \left[
    \begin{smallmatrix} a_2 \\ a_3 \\ \vdots \\ a_{d+1} \\ \end{smallmatrix}
    \right]
    -
    \left[
    \begin{smallmatrix} a_1 \\ a_2 \\ \vdots \\ a_d \end{smallmatrix}
    \right]
    \right)
    =
    \one - Cw
    -
    \left( 1 - \langle \one,w \rangle \right)
    \frac{1}{2}
    \left(
    \left[
    \begin{smallmatrix} a_1 \\ a_2 \\ \vdots \\ a_d \end{smallmatrix}
    \right]
    +
    \left[
    \begin{smallmatrix} a_2 \\ a_3 \\ \vdots \\ a_{d+1} \\ \end{smallmatrix}
    \right]
    \right).
\]
In order to understand which vectors $a$ lead to a large value $\lambda$, it is convenient to think of $a$ as (the discretization of) a function $y \colon [0,1] \to \R$, $w$ as a function $\omega \colon [0,1] \to \R$, and regard $d$ to be large.
We consider the continuous analogue
\begin{equation}
    \label{eqlukas}
    \lambda \cdot y'
    =
    1 - \operatorname{C} \omega
    -
    \left( 1 - \int_0^1 \omega(s) \, \mathrm{d}s \right)
    \cdot
    y.
\end{equation}
The matrix $C$ is replaced by the convolution operator $\operatorname{C}$ given by
\[
    (\operatorname{C} \omega)(t) := \int_0^1 y(t+s)\,\omega(s) \, \mathrm{d} s,
\]
where we set $y(t) = y(t-1)$ whenever $1 < t \leq 2$.
The boundary conditions $a_1 = 0$ and $a_{d+1} = 1$ are represented by $y(0) = 0$ and $y(1) = 1$.
Applying $\int_0^1 \cdot \, \mathrm{d} t$ to both sides of \eqref{eqlukas} yields
\[
    \lambda = 1 - \int_0^1 y(t) \, \mathrm{d} t \leq 1,
\]
which suggests that the general approach will not yield simplices of lattice width larger than $2d$.
Luckily, choosing $\omega$ to be a constant function already yields solutions close to that bound:
If $\omega$ is constant, taking the derivative of both sides in~\eqref{eqlukas}, we see that $y'' = \gamma \cdot y'$ holds for some $\gamma \in \R$, and hence $y$ has to be an exponential function.
In the previous section we have seen that, choosing the vector $a$ to be a discretization of an exponential function, we indeed obtain simplices of the desired lattice widths.

\section{Local maximizers in dimensions $4$ and $5$}
\label{secLocalMax}

Let us mention that it is possible to construct lattice-free simplices with slightly larger lattice widths than the simplices presented in the Section~\ref{secMain}.
In fact, by optimizing the choice of $a_1, \dots, a_{d + 1}$ to maximize $v_1 - v_{d + 1}$ in the construction from Section~\ref{secBackground} in dimensions $4$ and $5$, we obtain the following simplices.
For dimension $d=4$, we define the simplex $\Delta_4 := \conv \{\shift{v}{1},\dots,\shift{v}{5}\}$, where $v \in \R^5$ with
\begin{align*}
v_1 &= \frac{1}{5} \left( 7 - 2 \sqrt{5} + 2 \sqrt{10 + 2\sqrt{5}} \right) & v_4 &= v_2 \\
v_2 &= \frac{1}{5} \left( -3 + 4\sqrt{5} - 4 \sqrt{5 - 2\sqrt{5}} \right) & v_5 &= \frac{1}{5} \left( -3 - 2\sqrt{5} - 2\sqrt{5 + 2\sqrt{5}} \right). \\
v_3 &= \frac{1}{5} \left( 7 - 4\sqrt{5} + 6 \sqrt{5 - 2\sqrt{5}} \right) \\
\end{align*}
For dimension $d=5$, we define $\Delta_5 := \conv \{\shift{v}{1},\dots,\shift{v}{6}\}$, where $v \in \R^6$ with
\begin{align*}
v_1 &= \frac{1}{18} \left(57 - 7\sqrt{3} \right)   & v_4 &= v_3                                         \\
v_2 &= \frac{1}{3} \left(4 \sqrt{3} - 5 \right)    & v_5 &= v_2                                         \\
v_3 &= \frac{1}{18} \left(27 - 11 \sqrt{3} \right) & v_6 &= \frac{1}{18} \left( -33 -19 \sqrt{3} \right).   
\end{align*}
We can prove the following.
\begin{theorem}
Let $\pi_d: \R^{d + 1} \to \R^d$ denote the projection onto the first $d$ coordinates. Then, $\pi_4(\Delta_4)$ and $\pi_5(\Delta_5)$ are lattice-free simplices with
\[
    \lw(\pi_4(\Delta_4)) = 2 + 2 \sqrt{1 + \frac{2}{\sqrt{5}}}
    \quad \quad \text{and} \quad \quad
    \lw(\pi_5(\Delta_5)) = 5 + \frac{2}{\sqrt{3}}.
\]
Both of them are local maximizers of lattice width among lattice-free convex bodies in the respective dimensions.
\end{theorem}
A precise version of the latter statement is the following:
For $d=4$ and $d=5$ there is some $\varepsilon > 0$ such that every lattice-free convex body $K \subseteq \R^d$ whose Hausdorff distance to $\pi_d(\Delta_d)$ is at most $\varepsilon$ satisfies $\lw(K) \le \lw(\pi_d(\Delta_d))$.

The fact that $\pi_4(\Delta_4)$ and $\pi_5(\Delta_5)$ are lattice-free can be easily confirmed by calculating their inequality descriptions and using Lemma~\ref{lemLatticeFree1}.
The lattice widths can be determined using a strategy similar to the proof of Lemma \ref{lemLatticeWidth}.
Showing that both simplices are local maximizers can be conducted by directly following all steps used by Averkov, Codenotti, Macchia \& Santos in~\cite{averkov2021local} for the three-dimensional case.
As in~\cite{averkov2021local}, the necessary computations are rather complex but can be verified using a computer algebra system in a straightforward way.

\section{Open questions}
\label{secQuestions}

We conclude our paper by posing some questions that naturally arise from our result, starting with the following.

\spnewtheorem{prob}{Problem}{\bfseries}{}
\begin{prob}
    Are there any $d$-dimensional lattice-free convex bodies with lattice width greater than $2d$?
\end{prob}

Actually, we do not even know a $d$-dimensional lattice-free convex body with lattice width equal to $2d$.
As indicated in Section~\ref{secBackground}, we cannot exceed this bound with our approach.

Following the previous section, all known (local) maximizers of the lattice width among lattice-free convex bodies for $d \le 5$ are obtained by maximizing $v_1 - v_{d+1}$ over all vectors $a \in \R^{d+1}$, where $v$ is the unique solution of~\eqref{eqhusdok}.
The question directly arises whether this approach can be used to obtain local maximizers in dimensions $d \ge 6$.
Unfortunately, this does not work:
While the maximizing vectors $a$ are non-decreasing for $d \le 5$, and hence result in lattice-free simplices, this is not true anymore for $d \ge 6$.

\begin{prob}
    Determine local maximizers of the lattice width among all lattice-free convex bodies in dimensions $d \ge 6$.
\end{prob}

Recall that the known constructions in dimensions $3,4,5$ are only known to be \emph{local} maximizers.
We still do not know whether any of these simplices is actually a global maximizer.

\begin{prob}
    Do there exist any $d$-dimensional lattice-free convex bodies whose lattice widths exceed the lattice widths of the known local maximizers for $d=3,4,5$?
\end{prob}

\bibliographystyle{splncs04}
\bibliography{references}

\end{document}